\documentclass[12pt,twoside]{article}
\usepackage[latin9]{inputenc}
\usepackage{amsthm}
\usepackage{amstext}
\usepackage{amssymb}
\usepackage{esint}

\makeatletter
\newcommand{\lyxaddress}[1]{
\par {\raggedright #1
\vspace{1.4em}
\noindent\par}
}
\theoremstyle{plain}
\newtheorem{thm}{\protect\theoremname}
\theoremstyle{plain}
\newtheorem{lem}[thm]{\protect\lemmaname}
\theoremstyle{definition}
\newtheorem{example}[thm]{\protect\examplename}
\theoremstyle{definition}
\newtheorem{defn}[thm]{\protect\definitionname}
\theoremstyle{remark}
\newtheorem{rem}[thm]{\protect\remarkname}
\ifx\proof\undefined
\newenvironment{proof}[1][\protect\proofname]{\par
\normalfont\topsep6\p@\@plus6\p@\relax
\trivlist
\itemindent\parindent
\item[\hskip\labelsep
\scshape
#1]\ignorespaces
}{%
\endtrivlist\@endpefalse
}
\providecommand{\proofname}{Proof}
\fi

\makeatother

\providecommand{\definitionname}{Definition}
\providecommand{\examplename}{Example}
\providecommand{\lemmaname}{Lemma}
\providecommand{\remarkname}{Remark}
\providecommand{\theoremname}{Theorem}

\begin{document}

\title{Semiparametric testing of statistical functionals revisited}

\author{Vladimir Ostrovski}

\date{30.03.2012}
\maketitle
\begin{abstract}
Along the lines of Janssen's and Pfanzagl's work the testing theory
for statistical functionals is further developed for non-parametric
one-sample problems. Efficient tests for the one-sided and two-sided
problems are derived for nonparametric statistical functionals. The
asymptotic power function is calculated under implicit alternatives
and hypotheses, which are given by the functional itself, for the
one-sided and two-sided cases. Under mild regularity assumptions is
shown that these tests are asymptotic most powerful. The combination
of the modern theory of Le Cam and approximation in limit experiments
provide a deep insight into the upper bounds for asymptotic power
functions tests for the one-sided and two-sided problems of hypothesis
testing. As example tests concerning the von Mises functional are
treated in nonparametric context.
\end{abstract}

\lyxaddress{Dr Vladimir Ostrovski, Niederbeckstr., 23, 40472 Düsseldorf; email:
vladimir.ostrovski@web.de}

\section{Introduction}

Varios statistical problems can be formulate as problems of hypothesis
testing of the real-valued statistical functionals. For the non-parametric
statistics the theory of testing statistical functionals is the natural
point of view, which due to Pfanzagl, Wefelmayer (1982, 1985), Jansson
(1999, 2000) and some other. Moreover some already known important
tests can be understood as tests for statistical functionals. They
are studied exact with the general theory.

The statistical model is given by a set $\mathcal{P}\subset\mathcal{M}_{1}(\Omega,\mathcal{A})$
of relevant probability measures which is a subset of the set $\mathcal{M}_{1}(\Omega,\mathcal{A})$
of all probability measures on the some measure space $\Omega,\mathcal{A}$.
We observe $n\in\mathbb{N}$ independent identically distributed replication
$X_{1},\ldots,X_{n}$ of random variables which have a distribution
$P\in\mathcal{P}$. Let $\pi_{i}$ denote the $i$-th canonical projection.
Then we set $X_{i}:=\mathcal{L}\left(\pi_{i}|P^{n}\right)$. Let $k:\mathcal{P}\rightarrow\mathbb{R},P\mapsto k(P)$
be a real-valued statistical functional. A test problem is often given
by the one-sided hypothesis 
\begin{equation}
H_{1}=\left\{ P\in\mathcal{P}:k(P)\leq a\right\} \mbox{ against }K_{1}=\left\{ P\in\mathcal{P}:k(P)>a\right\} \label{one_sided_problem}
\end{equation}
or two-sided hypothesis 
\begin{equation}
H_{2}=\left\{ P\in\mathcal{P}:k(P)=a\right\} \mbox{ against }K_{1}=\left\{ P\in\mathcal{P}:k(P)\neq a\right\} \label{two_sided_problem}
\end{equation}
for some $a\in\left\{ k(P):P\in\mathcal{P}\right\} \subset\mathbb{R}.$
Here has to be noticed that not every test problem can be taken to
this form. For many purposes this is, however, sufficient. The treatment
of the more general statistical functionals and test problems is a
subject for future research.

An outline for this paper is as follows. Section 2 describes the differentiable
statistical functionals and the other basics. Section 2 provides implicit
alternatives and hypotheses which are given by the functional itself.
The tests for implicit hypotheses against implizit alternatives are
given. Section 3 studies asymptotic power functions of these tests.
Section 4 deals with the local asymptotic efficiency and optimality
of tests. Section includes proofs and concludes the paper.

\section{Differentiable statistical functionals}

Let be $P_{0}\in\mathcal{P}$ be fixed. The local geometry of the
set $\mathcal{P}$ of probability measures is discribed by all $L_{2}(P_{0})$-differentiable
curves $t\mapsto P_{t}$ and their tangents $g\in L_{2}(P_{0})$,
see Bickel et. al. (1993) and Strasser (1985,1998). A function $f:(-\varepsilon,\varepsilon)\rightarrow\mathcal{P},t\mapsto P_{t}$
for some $\varepsilon>0$ is called $L_{2}(P_{0})$-differentiable
curve in $\mathcal{P}$ with tangent $g\in L_{2}(P_{0})$ if 
\begin{equation}
\lim_{t\rightarrow0}\left\Vert \frac{2}{t}\left(\left(\frac{dP_{t}}{dP_{0}}\right)^{\frac{1}{2}}-1\right)-g\right\Vert _{L_{2}(P_{0})}=0,\label{l2_diff1}
\end{equation}
\begin{equation}
\lim_{t\rightarrow0}\frac{1}{t^{2}}P_{t}\left(\left\{ \frac{dP_{t}}{dP_{0}}=\infty\right\} \right)=0,\label{l2_diff2}
\end{equation}
there $\left\Vert g\right\Vert _{L_{2}(P_{0})}=\left(\int g^{2}dP_{0}\right)^{\frac{1}{2}}$
is the $L_{2}(P_{0})$-norm of $g$. Note that $\int gdP_{0}=0$ for
a tangent $g$ and 
\begin{equation}
\left.\frac{d}{dt}\int h\, dP_{t}\right|_{t=0}=\int hg\, dP_{0}\label{l2_diff3}
\end{equation}
for any bounded function $h:(\Omega,\mathcal{A})\rightarrow(\mathbb{R},\mathcal{B}),$
see Bickel et. al. (1993) or Strasser (1985) for more information
about $L_{2}(P_{0})$-differentiable curves. Set 
\[
L_{2}^{(0)}(P_{0})=\left\{ h\in L_{2}(P_{0}):\int hdP_{0}=0\right\} .
\]

There is the concept of the one-sided $L_{2}(P_{0})$-differentiable
curves $f:[0,\varepsilon)\rightarrow\mathcal{P},t\mapsto P_{t}$ which
assumes the one-sided limits $t\downarrow0$ in (\ref{l2_diff1})
and (\ref{l2_diff2}), see van der Vaart (1988). For our purposes
the one-sided $L_{2}(P_{0})$-differentiable curves are not necessary
because in the non-parametric statistics the set of the probability
measures is not not fixed exactly. Moreover the one-sided $L_{2}(P_{0})$-differentiable
curves can be extended to two-sided $L_{2}(P_{0})$-differentiable
curves. The following small lemma is useful for the work with the
$L_{2}(P_{0})$-differentiable curves.
\begin{lem}
\label{L2_diff_skalierung} Let $t\mapsto P_{t}$ be a $L_{2}(P_{0})$-differentiable
curve in $\mathcal{P}$ with tangent $g\in L_{2}(P_{0})$. For any
$a\in\mathbb{R}$ the curve $t\mapsto P_{at}$ is then $L_{2}(P_{0})$-differentiable
with tangente $ag$. 
\end{lem}
Let $K(P_{0},\mathcal{P})$ be a set of tangents of all $L_{2}(P_{0})$-differentiable
curves in $\mathcal{P}$. Lemma \ref{L2_diff_skalierung} implies,
that $K(P_{0},\mathcal{P})$ is a cone. The tangent space $T(P_{0},\mathcal{P})$
of $P_{0}$ in $\mathcal{P}$ is defined as a $L_{2}(P_{0})$-closure
of $K(P_{0},\mathcal{P})$. Tangent spaces was introduced by Pfanzagl
and Wefelmeyer(1982, 1985) and found an extended application in the
works of Janssen (1999, 2000). The concept of differentiable statistical
functionals due to Pfanzagl is motivated by the local linear approximation
of the differentiable functions, that is Taylor expansion of first
order. In addition assume that this approximation depends on the tangent
of $L_{2}(P_{0})$-differentiable curve linearly. A statistical functional
$k:\mathcal{P}\rightarrow\mathbb{R}$ is called differentiable at
$P_{0}$ with gradient $\dot{k}\in L_{2}^{(0)}(P_{0})$ if 
\begin{equation}
\lim_{t\downarrow0}\frac{1}{t}(k(P_{t})-k(P_{0}))=\int\dot{k}g\, dP_{0}\label{diff_func}
\end{equation}
for all $L_{2}(P_{0})$-differentiable curves $t\mapsto P_{t}$ in
$\mathcal{P}$ with tangent $g$. Note that the gradient $\dot{k}$
is not uniquely determined in general and it depend on $P_{0}$. The
one-sided limit in \ref{diff_func} does not represent any restriction.
Let $t\mapsto P_{t}$ be $L_{2}(P_{0})$-differentiable curve in $\mathcal{P}$
with tangent $g$. The curve $t\mapsto P_{-t}$ is $L_{2}(P_{0})$-differentiable
with tangent $-g$. By equation (\ref{diff_func}) we obtain 
\begin{equation}
\lim_{t\uparrow0}\frac{1}{t}\left(k\left(P_{t}\right)-k\left(P_{0}\right)\right)=\lim_{t\downarrow0}\frac{1}{\left(-t\right)}\left(k\left(P_{-t}\right)-k\left(P_{0}\right)\right)=\int\dot{k}g\, dP.
\end{equation}
The exists a unique canonical gradient $\widehat{k}\in T(P_{0}),\mathcal{P})$,
which has the smallest $L_{2}(P_{0})$-norm among all gradients. The
canonical gradient can be calculated as orthogonal projection of some
gradient $\dot{k}$ on the tangent space $T(P_{0},\mathcal{P})$.
\begin{example}
{[}von Mises functional{]} Let $h\in L_{1}(P)$ for all $P\in\mathcal{P}$.
The function $k:\mathcal{P}\rightarrow\mathbb{R},P\mapsto\int h\, dP$
is called von Mises functional. Let $t\mapsto P_{t}$ be $L_{2}(P_{0})$-differentiable
curve in $\mathcal{P}$ with tangent $g$. If $E_{P_{0}}(h^{2})<infty$
and $\limsup_{t\downarrow0}E_{P_{t}}(h^{2})$ hold then we have 
\begin{equation}
\left.\frac{d}{dt}k(P_{t})\right|_{t=0}=\int gh\, dP_{0}=\int g(h-E_{P_{0}})(h))dP_{0},\label{mises_diff}
\end{equation}
see Bickel et. al. (1993), p. 457, Proposition 2 for the proof. 

The conditions for the differentiability can be formulated in the
non-parametric context by means of Hellinger distance. Let $d(P,Q)$
denote the Hellinger distance of probability measures $P,Q\in\mathcal{M}_{1}(\Omega,\mathcal{A})$,
see \cite{Strasser:85} for the definition of the Hellinger distance.

\label{top_equi} The following three statements are equivalent: 

(a) $\limsup_{n\rightarrow\infty}E_{P_{n}}\left(h^{2}\right)<\infty$
for all sequences $\left(P_{n}\right)_{n\in\mathbb{N}}$ in $\mathcal{P}$
such that $\lim\limits _{n\rightarrow\infty}d\left(P_{n},P_{0}\right)=0.$

(b) There exists $\varepsilon>0$ such that $\sup\{E_{P}\left(h^{2}\right):P\in\mathcal{P}\mbox{ with }d(P,P_{0})<\varepsilon\}<\infty.$

(c) There exist $\varepsilon>0$ and $K>0$ such that $E_{P}\left(h^{2}\right)<K$
for all $P\in\mathcal{P}$ with $d\left(P,P_{0}\right)<\varepsilon$. 

By (\ref{mises_diff}) the functional $k$ is differentiable at $P_{0}$
if one of the statements of Lemma \ref{top_equi} is fulfilled. The
function $h-E_{P_{0}}$ is then a gradient of $k$ at $P_{0}$ and
$h-E_{P_{0}}\in L_{2}^{(0)}(P_{0})$. If $h-E_{P_{0}}\in T(P_{0},\mathcal{P})$
holds then $h-E_{P_{0}}$ is already the canonical gradient. For most
situations of the non-parametric statistics it can be assumed that
the condition $h-E_{P_{0}}\in T(P_{0},\mathcal{P})$ is satisfied
at all $P_{0}\in\mathcal{P}.$ 
\end{example}

\begin{example}
{[}Median functional{]} Let $\mathcal{P}\subset\mathcal{M}_{1}(\mathbb{R},\mathcal{B})$
be a set of probability measures with connected support and strictly
positive Lebesgue density on its support. Let $\mbox{med(P)}$ be
the unique defined median of $P$. If the Lebesgue density $f_{P_{0}}$
of $P_{0}\in\mathcal{P}$ is continuous in the some neighbourhood
of $\mbox{med}(P_{0})$, then the functional $k:\mathcal{P}\rightarrow\mathbb{R},P\mapsto\mbox{med(P)}$
is differentiable at $P_{0}$ with gradient 
\begin{equation}
\dot{k}(x)=\left(2f_{P_{0}}(\mbox{med}(P_{0}))\right)^{-1}\mbox{sign}\left(F_{P_{0}}(x)-\frac{1}{2}\right)\label{med_grad}
\end{equation}
where $F_{P_{0}}$ denotes the distribution function of $P_{0}$,
see Pfanzagl and Wefelmayer (1985), p. 150, Proposition 5.5.4 for
the proof. 
\end{example}

\begin{example}
{[}multinomial distributions{]} Let $\Omega={a_{1},\ldots,a_{m}}$
be finite set for some $m\in\mathbb{N}$ and $\mathcal{A}$ be the
power set of $\Omega$. Let $\mathcal{P}=\mathcal{M}_{1}(\Omega,\mathcal{A})$.
Any statistical functional $k:\mathcal{P}\rightarrow\mathbb{R}$ has
a representation 
\[
k(P)=f(P(\{a_{1}\}),\ldots,P(\{a_{m}\}))
\]
for some $f:[0,1]^{m}\rightarrow\mathbb{R}$. Let $t\mapsto P_{t}$
be $L_{2}(P_{0})$-differentiable curve in $\mathcal{P}$ with tangent
$g$. Assume that $f$ is differentiable at $(p_{1},\ldots,p_{m}):=(P_{0}(\{a_{1}\}),\ldots,P_{0}(\{a_{m}\}))$
and denote $w_{i}=\left.\frac{d}{dx_{i}}f(x_{1},\ldots,x_{m})\right|_{(p_{1},\ldots,p_{m})}$.
By (\ref{l2_diff3}) we obtain 
\begin{equation}
\left.\frac{d}{dt}k(P_{t})\right|_{t=0}=\sum_{i=1}^{m}w_{i}\int\mathbf{1}_{\{a_{i}\}}g\, dP_{0}=\int g\left(\sum_{i=1}^{m}w_{i}\mathbf{1}_{\{a_{i}\}}\right)dP_{0}.\label{multi_gradient}
\end{equation}
Thus the functional $k$ is differentiable at $P_{0}$ with the gradient
\begin{equation}
\dot{k}=\sum_{i=1}^{m}w_{i}\mathbf{1}_{\{a_{i}\}}-E_{P_{0}}\left(\sum_{i=1}^{m}w_{i}\mathbf{1}_{\{a_{i}\}}\right)=\sum_{i=1}^{m}w_{i}(\mathbf{1}_{\{a_{i}\}}-p_{i}),
\end{equation}
which depends on $P_{0}$. 
\end{example}

\section{Implicit alternatives and hypotheses }

We fix a particilar $P_{0}\in H_{2}$ and assume further that a functional
$k:\mathcal{P}\rightarrow\mathbb{R}$ is differentiable at $P_{0}$
with the canonical gradient $\widetilde{k}\in T(P_{0},\mathcal{P})$
and 
\begin{equation}
\int\widetilde{k}^{2}dP_{0}\neq0.\label{not_null}
\end{equation}
The assumption (\ref{not_null}) is necessary. However, there are
some important cases in nonparametric statistics such that it is not
fulfilled. Let $k:\mathcal{P}\rightarrow\mathbb{R}$ be a some distance
between $P$ and $P_{0}.$ The following two cases are then possible.
Either the statistical functional $k$ is not differentiable at $P_{0}$
or the canonical gradient $\widetilde{k}$ equals $0\in L_{2}(P_{0})$.

We keep in mind that the canonical gradient $\widetilde{k}$ depends
on $P_{0}$ in general. For the one-sided and two-sided test problems
we look for local parametrization which depends only on the statistical
functional $k$. The basic idea of the implicit alternatives due to
Pfanzagl and Wefelmayer (1982, Section 8; 1985, Section 6). The reader
finds a discussion and interpretation about this in Janssen (1999a,
1999b). The implicit alternatives and hypotheses are represented formalized
here and generalized a little. 
\begin{defn}
\label{set_f} The set $\mathcal{F}$ of all imlizit alternatives
and hypotheses contains all sequences $(P_{t_{n}})_{n\in\mathbb{N}}$
of probability measures which satisfy the following requirements: 

The sequence $t_{n}\in[0,\infty)$ fulfills $\lim_{n\rightarrow\infty}t_{n}\sqrt{n}>0$. 

There exist some $\varepsilon>0$ and $L_{2}(P_{0})$-differentiable
curve $f:(-\varepsilon,\varepsilon)\rightarrow\mathcal{P},t\mapsto P_{t}$
such that $f(t_{n})=P_{t_{n}}$ holds for all $n\in\mathbb{N}$ 
\end{defn}

\begin{defn}
Implicit alternatives and hypotheses for $H_{1}$ against $K_{1}${]}
The set $\mathcal{K}_{1}$ of the implicit alternatives is given by
The set $\mathcal{H}_{1}$ of the implicit hypotheses is given by
\[
\mathcal{H}_{1}:=\left\{ (P_{t_{n}})_{n\in\mathbb{N}}\in\mathcal{F}:\lim_{n\rightarrow\infty}\sqrt{n}(k(P_{t_{n}})-k(P_{0}))\leq0\right\} .
\]

\end{defn}
In Definition \ref{set_f} would suffice if a sequence $(P_{t_{n}})_{n\in\mathbb{N}}$
lies in the set $\mathcal{P}$. The appropriate $L_{2}(P_{0})$-differentiable
curve does not need to lie in $\mathcal{P}$. However the established
elegant definion is chosen here, see also Janssen(1999 a,b). In this
case the set $K(P_{0},\mathcal{P})$ of tangents describes the local
properties of the sets $\mathcal{K}_{1}$ of all implicit alternatives
and $\mathcal{H}_{1}$ of all implicit hypotheses. Hence the generalization
of the definition is renounced. In addition, the difference is not
really significant because in the non-parametric statistics the set
$\mathcal{P}$ is not fixed exactly enough. The local modeling happens
much more via the set $K(P_{0},\mathcal{P})$ of tangents. The one-sided
test problem $H_{1}$ against $K_{1}$ is asymptotically equivalent
to the test problem $\mathcal{H}_{1}$ against $K(P_{0},\mathcal{P})$
local in the neighbourhood of $P_{0}$, see Janssen(1999 a,b) for
details. For the two-sided test problem $H_{2}$ against $K_{2}$
we obtain analogously the following implicit alternatives and hypotheses. 
\begin{defn}
{[}Implicit alternatives hypotheses for $H_{2}$ against $K_{2}${]}
The set $\mathcal{K}_{2}$ of the implicit alternatives is given by
\[
\mathcal{K}_{2}:=\left\{ (P_{t_{n}})_{n\in\mathbb{N}}\in\mathcal{F}:\lim_{n\rightarrow\infty}\sqrt{n}(k(P_{t_{n}})-k(P_{0}))\neq0\right\} .
\]
The set $\mathcal{H}_{2}$ of the implicit hypotheses is given by
\[
\mathcal{H}_{2}:=\left\{ (P_{t_{n}})_{n\in\mathbb{N}}\in\mathcal{F}:\lim_{n\rightarrow\infty}\sqrt{n}(k(P_{t_{n}})-k(P_{0}))=0\right\} .
\]

\end{defn}
We solve the local test problems $\mathcal{H}_{1}$ against $\mathcal{K}_{1}$
and $\mathcal{H}_{2}$ against $\mathcal{K}_{2}$. Thus we derive
the sequences of the tests for local test problems, which usually
depends on $P_{0}$. If there exists a distribution-free form of these
tests then the original one- and two-sided test problems are solved
too. The construction of the distribution-free tests is not an aim
of this work. There exist an extensive literature about it, see Hájek
et al. (1999) and Janssen (1999,2000) for example.

\section{Tests and their asymptotic power functions}

Let $\left\Vert h\right\Vert :=\left(\int h^{2}dP_{0}\right)^{\frac{1}{2}}$
denote the $L_{2}(P_{0})$-norm of $h\in L_{2}(P_{0})$. A meaningful
test for $\mathcal{H}_{1}$ against $\mathcal{K}_{1}$ is defined
by 
\begin{equation}
\varphi_{1n}(X_{1},\ldots,X_{n})=\mathbf{1}_{\left(u_{1-\alpha}\left\Vert \widetilde{k}\right\Vert ,+\infty\right)}\left(\frac{1}{\sqrt{n}}\sum_{i=1}^{n}\widetilde{k}(X_{i})\right)\label{one-sided-test}
\end{equation}
and a reasonable test for $\mathcal{H}_{2}$ against $\mathcal{K}_{2}$
is given by 
\begin{equation}
\varphi_{2n}(X_{1},\ldots,X_{n})=\mathbf{1}_{\left(u_{1-\frac{\alpha}{2}}\left\Vert \widetilde{k}\right\Vert ,+\infty\right)}\left(\left|\frac{1}{\sqrt{n}}\sum_{i=1}^{n}\widetilde{k}(X_{i})\right|\right),\label{two-sided-test}
\end{equation}
where $u_{1-\alpha}$ is the $1-\alpha$ quantile of the standard
normal distribution. The detailed motivation of these tests can be
found in Pfanzagl and Wefelmayer(1982, 1985) and Jansson(1999 a,b).
Next we compute the asymptotic power function of tests $\varphi_{1n}$
and $\varphi_{2n}$ along implicit alternatives and hypotheses exactly. 
\begin{thm}
\label{asymp_power_function} For each sequence $(P_{t_{n}})_{n\in\mathbb{N}}\in\mathcal{F}$
we have 
\begin{equation}
\lim_{n\rightarrow\infty}E_{P_{t_{n}}^{n}}(\varphi_{1n})=\Phi\left(\frac{\vartheta}{\|\widetilde{k}\|}-u_{1-\alpha}\right),\label{one_sided_power}
\end{equation}
\begin{equation}
\lim_{n\rightarrow\infty}E_{P_{t_{n}}^{n}}(\varphi_{2n})=\Phi\left(\frac{\vartheta}{\|\widetilde{k}\|}-u_{1-\frac{\alpha}{2}}\right)+\Phi\left(-\frac{\vartheta}{\|\widetilde{k}\|}-u_{1-\frac{\alpha}{2}}\right),\label{two_sided_power}
\end{equation}
where $\vartheta:=\lim_{n\rightarrow\infty}\sqrt{n}(k(P_{t_{n}})-k(P_{0}))$
is an additional local parameter. 
\end{thm}
Equation (\ref{one_sided_power}) of theorem \ref{asymp_power_function}
coincides with Janssen(1999a), theorem 3.1(a). The proof of theorem
\ref{asymp_power_function} follows the canonical lines of Le Cam's
theory and is presented in section \ref{proofs}. 

The results of theorem \ref{asymp_power_function} have some interesting
consequences for tests $\varphi_{1n}$ and $\varphi_{2n}$. 
\begin{itemize}
\item \emph{Local asymptotic unbiasedness:} For each implicit hypothesis
$(P_{t_{n}})_{n\in\mathbb{N}}\in\mathcal{H}_{i}$ we have $lim_{n\rightarrow\infty}E_{P_{t_{n}}^{n}}(\varphi_{in})\leq\alpha$
for $i=1,2.$ 
\item \emph{Level $\alpha$ property:} For each implicit alternative $(P_{t_{n}})_{n\in\mathbb{N}}\in\mathcal{K}_{i}$
we get $lim_{n\rightarrow\infty}E_{P_{t_{n}}^{n}}(\varphi_{in})\geq\alpha$
for $i=1,2.$ 
\item \emph{Independence on the $L_{2}(P_{0})$-differentiable curves:}
The asymptotic power functions (\ref{one_sided_power}) and (\ref{two_sided_power})
depend only on a local parameter $\vartheta\in\mathbb{R}$ and the
$L_{2}(P_{0})$-norm $\|\widetilde{k}\|$ of the canonical gradient.
They do not depend on the underlying $L_{2}(P_{0})$-differentiable
curve immediately. 
\end{itemize}

\section{Efficiency}

In this section we present the main results to the asymptotic efficiency
of the tests $\varphi_{1n}$ and $\varphi_{2n}$. For the readers
convenience we begin with the preliminary results, which are partially
known, see Pfanzagl and Wefelmayer (1982, 1985) and Strasser (1985).
The proof is a bit different as in the given references. 
\begin{defn}
The sequence $\phi_{n}$ of tests for $\mathcal{H}_{1}$ against $\mathcal{K}_{1}$
has the level $\alpha$ property asymptotically if $\lim_{n\rightarrow\infty}\int\phi_{n}dP_{t_{n}}^{n}=\alpha$
for all $(P_{t_{n}})_{n\in\mathbb{N}}\in\mathcal{H}_{2}.$ 

The sequence $\phi_{n}$ of tests for $\mathcal{H}_{2}$ against $\mathcal{K}_{2}$
is asymptotically unbiased if $\lim_{n\rightarrow\infty}\int\phi_{n}dP_{t_{n}}^{n}\geq\alpha$
for all $(P_{t_{n}})_{n\in\mathbb{N}}\in\mathcal{K}_{2}$ and $\lim_{n\rightarrow\infty}\int\phi_{n}dP_{t_{n}}^{n}\leq\alpha$
for all $(P_{t_{n}})_{n\in\mathbb{N}}\in\mathcal{H}_{2}$. \end{defn}
\begin{thm}
\label{one_two_sided_easy} Suppose that $T(P_{0},\mathcal{P})=K(P_{0},\mathcal{P}).$ 

(a) The sequence $\varphi_{1n}$ of tests is then asymptotically most
powerfull in the set of all sequences of tests for $\mathcal{H}_{1}$
against $\mathcal{K}_{1}$, which have level $\alpha$ property asymptotically. 

(b) The sequence $\varphi_{2n}$ of tests is then asymptotically most
powerfull in the set of all asymptotically unbiased sequences of tests
for $\mathcal{H}_{2}$ against $\mathcal{K}_{2}.$ 
\end{thm}
The short but comprehensible proof of theorem \ref{one_two_sided_easy}
is given in section \ref{proofs}. Further we need the following Definition
for the conveniently notation.
\begin{defn}
 The set $N(\widetilde{k},P_{0},\mathcal{P}):=\{h\in T(P_{0},\mathcal{P}):\int h\widetilde{k}\, dP_{0}=0\}$
is the orthogonal complement of the canonical gradient $\widetilde{k}$
in the tangential space $T(P_{0},\mathcal{P}).$ 
\end{defn}
Thus we introduce the main results. The assumption $T(P_{0},\mathcal{P})=K(P_{0},\mathcal{P})$
of theorem \ref{one_two_sided_easy} can be relaxed for the one-sided
and for the two-sided tests. The proof of the following two theorems
use an asymptotic representation theorem of van der Vaart (1991).
The following theorem is the main result for the one-sided tests $\varphi_{1n}$
for $\mathcal{H}_{1}$ against $\mathcal{K}_{1}.$ 
\begin{thm}
\label{result_1} Suppose, that the set $K(P_{0},\mathcal{P})\cap N(\widetilde{k},P_{0},\mathcal{P})$
is dense in $N(\widetilde{k},P_{0},\mathcal{P})$ with respect to
$L_{2}(P_{0})$-norm. The sequence $\varphi_{1n}$ of tests is then
asymptotically most powerfull in the set of all sequences of tests
for $\mathcal{H}_{1}$ against $\mathcal{K}_{1}$, which have level
$\alpha$ property asymptotically. 
\end{thm}
Theorem \ref{result_1} generalizes theorem 3.1 and corollary 3.3
of Janssen (1999a), since the main assumption of Janssen (1999a) is
quite stronger as the assumption of theorem \ref{result_1}. This
will be discussed in remark \ref{konvex_fertig} below. The proof
of theorem \ref{result_1} differs explicitly from the proofs in Janssen
(1999a). The theorem \ref{result_2} below is the main result for
the two-sided tests $\varphi_{2n}$ for $\mathcal{H}_{2}$ against
$\mathcal{K}_{2}.$ 
\begin{thm}
\label{result_2} Suppose, that the set $K(P_{0},\mathcal{P})\cap N(\widetilde{k},P_{0},\mathcal{P})$
is dense in $N(\widetilde{k},P_{0},\mathcal{P})$ and $K(P_{0},\mathcal{P})$
is dense in $T(P_{0},\mathcal{P})$ with respect to $L_{2}(P_{0})$-norm.
The sequence $\varphi_{2n}$ of tests is then asymptotically most
powerfull in the set of all asymptotically unbiased sequences of tests
for $\mathcal{H}_{2}$ against $\mathcal{K}_{2}.$ 
\end{thm}
The proofs of theorems \ref{result_1} and \ref{result_2} are given
in section \ref{proofs}. 
\begin{rem}
\label{konvex_fertig} We show, that the assumption of theorems \ref{result_1}
and \ref{result_2} are already satisfied if $K(P_{0},\mathcal{P})$
is a vector space. The set $K(P_{0},\mathcal{P})$ is a vector space
iff it is convex since $K(P_{0},\mathcal{P})$ is a cone. If $K(P_{0},\mathcal{P})$
is a vector space then it is obviously dense in tangential space $T(P_{0},\mathcal{P})$,
see definition of $T(P_{0},\mathcal{P})$. The set $W:=K(P_{0},\mathcal{P})\cap N(\widetilde{k},P_{0},\mathcal{P})$
is then a vector space. Let $h\in N(\widetilde{k},P_{0},\mathcal{P})$
be fixed. There exists some sequence $(h_{n})_{n\in\mathbb{N}}$ in
$K(P_{0},\mathcal{P})$ such that.
\end{rem}

\section{Proofs}
\begin{proof}
\emph{of Lemma \ref{L2_diff_skalierung}}

The case $a=0$ is trivial. Let $a\neq0$. The condition (\ref{l2_diff2})
is fulfilled since 
\[
\lim_{t\rightarrow0}\frac{1}{t^{2}}P_{at}\left(\left\{ \frac{dP_{at}}{dP_{0}}=\infty\right\} \right)=a^{2}\lim_{t\rightarrow0}\frac{1}{t^{2}}P_{t}\left(\left\{ \frac{dP_{t}}{dP_{0}}=\infty\right\} \right)=0.
\]
The condition (\ref{l2_diff1}) is satisfied because 
\begin{eqnarray*}
\textrm{} &  & \lim_{t\rightarrow0}\left\Vert \frac{2}{t}\left(\left(\frac{dP_{at}}{dP_{0}}\right)^{\frac{1}{2}}-1\right)-ag\right\Vert _{L_{2}(P_{0})}\\
 & = & |a|\lim_{s\rightarrow0}\left\Vert \frac{2}{s}\left(\left(\frac{dP_{s}}{dP_{0}}\right)^{\frac{1}{2}}-1\right)-ag\right\Vert _{L_{2}(P_{0})}=0.
\end{eqnarray*}

\end{proof}

\begin{proof}
Lemma \ref{top_equi} The equivalence between (b) and (c) is easy
to show.\\
 $(a)\Rightarrow(b)$ Supposed the statement (b) is wrong. Then there
exists a sequence $P_{n}\in\mathcal{P}$ such that $\lim_{n\rightarrow\infty}d(P_{n},P_{0})=0$
and $\lim_{n\rightarrow\infty}E_{P_{n}}(h^{2})=\infty.$ That implies
$\limsup_{n\rightarrow\infty}E_{P_{n}}(h^{2})=\infty.$ This is contradiction
to $(a).$ \\
 $(b)\Rightarrow(a)$ Let $P_{n}\in\mathcal{P}$ be a sequence in
$\mathcal{P}$ with $\lim_{n\rightarrow0}d(P_{n},P_{0})=0.$ For any
$\varepsilon>0$ there exists $n_{0}\in\mathbb{N}$ such that $d(P_{n},P_{0})<\varepsilon$
for all $n\geq n_{0}$. Consequently we have 
\[
\limsup_{n\rightarrow\infty}E_{P_{n}}(h^{2})<\sup\{E_{P}(h^{2}):P\in\mathcal{P}\mbox{ with }d(P,P_{0})<\varepsilon\}<\infty.
\]

\end{proof}

\begin{proof}
\emph{of Theorem \ref{asymp_power_function}} Let $(P_{t_{n}})_{n\in\mathbb{N}}\in\mathcal{F}$
and let $f:t\mapsto P_{t}$ be an associated $L_{2}(P_{0})$-differentiable
curve in $\mathcal{P}$ with tangent $g\in L_{2}(P_{0})$ such that
$f(t_{n})=P_{t_{n}}$ for all $n\in\mathbb{N}$. At first we calculate
a local parameter 
\begin{equation}
\vartheta=\lim_{n\rightarrow\infty}\sqrt{n}(k(P_{t_{n}})-k(P_{0}))=t\int g\widetilde{k}\, dP_{0},\label{local_parameter}
\end{equation}
where $t:=\lim_{n\rightarrow\infty}n^{\frac{1}{2}}t_{n}.$ Assume,
that $\int g^{2}\, dP_{0}=0$. By (\ref{local_parameter}) we obtain
$\vartheta=0.$ Let $|P-Q|$ denote the variational distance between
probability measures $P$ and $Q$ on the same probability space,
see Strasser(1985), p. 5. It is well known, that $\lim_{n\rightarrow\infty}|P_{t_{n}}^{n}-P_{0}^{n}|=0$,
see Strasser(1985), p.386, theorem 75.8. In this case we have 
\[
\left|\int\varphi_{in}dP_{t_{n}}^{n}-\int\varphi_{in}dP_{0}^{n}\right|\leq2|P_{t_{n}}^{n}-P_{0}^{n}|\rightarrow0
\]
for $i=1,2$. This implies 
\[
\lim_{n\rightarrow\infty}\int\varphi_{in}dP_{t_{n}}^{n}=\lim_{n\rightarrow\infty}\int\varphi_{in}dP_{0}^{n}=\alpha
\]
for $i=1,2$, where the last equality follows from central limit theorem.
Thus, the statement of the theorem is proved for the case $\int g^{2}\, dP_{0}=0$.\\
 Further we assume $\int g^{2}dP_{0}\neq0.$ The sequence $(P_{t_{n}})_{n\in\mathbb{N}}$
of probability measures is local asymptotic normal with central sequence
$\sum_{i=1}^{n}t_{i}g(X_{i})$, see Le Cam and Yang(2000), p. 177,
section 7.2. The central limit theorem implies 
\[
\mathcal{L}\left(\left.\left(\frac{1}{\sqrt{n}}\sum_{i=1}^{n}\widetilde{k}(X_{i}),\sum_{i=1}^{n}t_{i}g(X_{i})\right)\right|P_{0}^{n}\right)\rightarrow N\left(\left(\begin{array}{c}
0\\
0
\end{array}\right),\left(\begin{array}{cc}
\sigma_{1}^{2} & \sigma_{12}\\
\sigma_{12} & \sigma_{2}^{2}
\end{array}\right)\right),
\]
where $\sigma_{1}^{2}=\left\Vert \widetilde{k}\right\Vert ^{2},$
$\sigma_{12}=\vartheta$ and $\sigma_{2}^{2}=t^{2}\int g^{2}dP_{0}.$
Le Cam's third lemma now implies 
\[
\mathcal{L}\left(\left.\frac{1}{\sqrt{n}}\sum_{i=1}^{n}\widetilde{k}(X_{i})\right|P_{t_{n}}^{n}\right)\rightarrow N\left(\sigma_{12},\sigma_{2}^{2}\right)=N\left(\vartheta,\|\widetilde{k}\|^{2}\right),
\]
see Hájek et. al (1999), p. 257. Hence we obtain 
\[
\lim_{n\rightarrow\infty}\int\varphi_{1n}dP_{t_{n}}=1-\Phi\left(\frac{u_{1-\alpha}\|\widetilde{k}\|-\vartheta}{\|\widetilde{k}\|}\right)=\Phi\left(\frac{\vartheta}{\|\widetilde{k}\|}-u_{1-\alpha}\right),
\]
\[
\lim_{n\rightarrow\infty}\int\varphi_{2n}dP_{t_{n}}=1-\Phi\left(\frac{u_{1-\frac{\alpha}{2}}\|\widetilde{k}\|-\vartheta}{\|\widetilde{k}\|}\right)+\Phi\left(-\frac{u_{1-\frac{\alpha}{2}}\|\widetilde{k}\|-\vartheta}{\|\widetilde{k}\|}\right).
\]

\end{proof}
The proof of theorem \ref{one_two_sided_easy} needs some preparations.
For the application of the theory of the Gaussian shift experiments
we need the explicit assignment between the probability measures from
$\mathcal{P}$ and tangents from $T(P_{0},\mathcal{P})$. Hence we
construct a semiparametric family of probability measures which parameter
space is equal to $T(P_{0},\mathcal{P})$. The following construction
of the $L_{2}(P_{0})$-differentiable curves corresponds to the normal
tangent master model of Janssen(2004). 
\begin{lem}
\label{konstruktion} The function $f_{g}:=c(g)^{-1}\left(1+\frac{1}{2}g\right)^{2}$
with $c(g):=\int\left(1+\frac{1}{2}g\right)^{2}dP_{0}$ is then a
density of the probability measures $P_{g}:=f_{g}P_{0}$ for all $g\in L_{2}^{(0)}(P_{0})$.
Moreover for each $g\in L_{2}^{(0)}(P_{0})$ the curve $\mathbb{R}\rightarrow\mathcal{M}_{1}(\Omega,\mathcal{A}),t\mapsto P_{tg}$
is $L_{2}(P_{0})$-differentiable with tangent $g.$ \end{lem}
\begin{proof}
The function $f_{g}$ is well-defined for all $g\in L_{2}^{(0)}(P_{0})$
since 
\begin{equation}
c(g):=\int\left(1+\frac{1}{2}g\right)^{2}dP_{0}=1+\frac{t^{2}}{4}\int g^{2}dP_{0}\geq1.\label{c(g)}
\end{equation}
The function $f_{g}$ is also the $P_{0}$-density of some probability
measures $P_{g}:=f_{g}P_{0}.$ We show now that the curve $\mathbb{R}\rightarrow\mathcal{M}_{1}(\Omega,\mathcal{A}),t\mapsto P_{tg}$
is differentiable with tangent $g.$ The condition \ref{l2_diff2}
is fulfilled because of $P_{tg}\ll P_{0}.$ The condition \ref{l2_diff1}
remains to prove. We obtain at first $\left.\frac{d}{dt}c(tg)\right|_{t=0}=0$
and $\frac{2}{t}\left(|1+\frac{1}{2}tg|-1\right)\rightarrow g$ almost
everywhere under $P_{0}$ for $t\rightarrow0.$ By the low of the
dominated convergence we have 
\begin{equation}
\lim_{t\rightarrow0}\left\Vert \frac{2}{t}\left(\left|1+\frac{1}{2}tg\right|-1\right)-g\right\Vert _{L_{2}(P_{0})}=0\label{hin1}
\end{equation}
since $\left|\frac{2}{t}\left(\left|1+\frac{1}{2}tg\right|-1\right)\right|\leq\frac{2}{|t|}\left|1+\frac{1}{2}tg-1\right|=|g|.$
In addition we obtain 
\begin{equation}
\lim_{t\rightarrow0}\frac{c(tg)^{-\frac{1}{2}}-1}{t}=\left.\frac{d}{dt}c(tg)^{-\frac{1}{2}}\right|_{t=0}=\left.-\frac{1}{2}c(tg)^{-\frac{3}{2}}\frac{d}{dt}c(tg)\right|_{t=0}=0.\label{hin2}
\end{equation}
By (\ref{hin1}) and (\ref{hin2}) together we conclude 
\begin{eqnarray*}
 &  & \left\Vert \frac{2}{t}\left(\left(\frac{dP_{tg}}{dP_{0}}\right)^{\frac{1}{2}}-1\right)-g\right\Vert _{L_{2}\left(P_{0}\right)}\\
 & = & \left\Vert \frac{2}{t}\left(\left|1+\frac{1}{2}tg\right|c(tg)^{-\frac{1}{2}}-1\right)-g\right\Vert _{L_{2}\left(P_{0}\right)}\\
 & = & \left\Vert c(tg)^{-\frac{1}{2}}\left(\frac{2}{t}\left(\left|1+\frac{t}{2}g\right|-1\right)-g\right)+\frac{2}{t}\left(c(tg)^{-\frac{1}{2}}-1\right)\right.\\
 &  & \left.+g\left(c(tg)^{-\frac{1}{2}}-1\right)\right\Vert _{L_{2}\left(P_{0}\right)}\\
\textrm{} & \leq & c(tg)^{-\frac{1}{2}}\left\Vert \frac{2}{t}\left(\left|1+\frac{1}{2}tg\right|-1\right)-g\right\Vert _{L_{2}\left(P_{0}\right)}+\left|\frac{2}{t}\left(c(tg)^{-\frac{1}{2}}-1\right)\right|\\
 &  & +\left\Vert g\right\Vert _{L_{2}\left(P_{0}\right)}\left|c(tg)^{-\frac{1}{2}}-1\right|\\
 & \rightarrow & 0\,\mbox{ for }\, t\rightarrow0.
\end{eqnarray*}
This statement concludes the proof. 
\end{proof}
Define $P_{n,g}:=P_{n^{-\frac{1}{2}}g}$ for all $g\in T(P_{0},\mathcal{P})$
and $n\in\mathbb{N}.$ Recall that $T(P_{0},\mathcal{P})$ is a Hilbert
space with scalar product $(h,g)\mapsto\int hg\, dP_{0}.$ We consider
the sequence $E_{n}:=(\Omega^{n},\mathcal{A}^{n},\left\{ P_{n,g}:g\in T(P_{0},\mathcal{P})\right\} )$
of the statistics experiments.
\begin{lem}
The sequence $E_{n}$ converges weakly to a Gaussian shift on $H.$ \end{lem}
\begin{proof}
We use Theorem 80.2 of Strasser(1985), p.409 for the proof. From Le
Cam and Yang(2000), p. 177, section 7.2 we obtain the expansion 
\begin{equation}
\log\left(\frac{P_{n,g}}{P_{n,0}}\right)=\frac{1}{\sqrt{n}}\sum_{i=1}^{n}g(X_{i})-\frac{1}{2}\int g^{2}\, dP_{0}+R_{n,g},
\end{equation}
where $\lim_{n\rightarrow\infty}P_{n,0}(|R_{n,g}|>\varepsilon)=0$
for all $\varepsilon>0.$ This expansion implies 
\begin{equation}
L_{n}(g):=\log\left(\frac{P_{n,g}}{P_{n,0}}\right)+\frac{1}{2}\int g^{2}dP_{0}=\frac{1}{\sqrt{n}}\sum_{i=1}^{n}g(X_{i})+R_{n,g}.
\end{equation}
Thus we have the weak convergency $\mathcal{L}(L_{n}(g)|P_{n,0})\rightarrow N(0,\int g^{2}\, dP_{0})$
and for all $a,b\in\mathbb{R},$ $g_{1},g_{2}\in T(P_{0},\mathcal{P})$
and $\varepsilon>0$ we obtain 
\begin{eqnarray*}
 &  & P_{n,0}(|aL_{n}(g_{1})+bL_{n}(g_{2})-L_{n}(ag_{1}+bg_{2})|>\varepsilon)\\
 & = & P_{n,0}(|aR_{n,g_{1}}+bR_{n,g_{2}}+R_{n,ag_{1}+bg_{2}}|>\varepsilon)\\
 & \rightarrow & 0\,\mbox{ for }\, n\rightarrow\infty
\end{eqnarray*}
This conclude the proof. 
\end{proof}

\begin{proof}
\emph{of Theorem \ref{one_two_sided_easy}} Let $(P_{t_{n}})_{n\in\mathbb{N}}\in\mathcal{F}$
and let $f:t\mapsto P_{t}$ be an associated $L_{2}(P_{0})$-differentiable
curve in $\mathcal{P}$ with tangent $g\in L_{2}(P_{0})$ such that
$f(t_{n})=P_{t_{n}}$ for all $n\in\mathbb{N}$. For $t:=\lim_{n\rightarrow\infty}n^{\frac{1}{2}}t_{n}$
we obtain 
\begin{equation}
\vartheta=\lim_{n\rightarrow\infty}\sqrt{n}(k(P_{t_{n}})-k(P_{0}))=t\int g\widetilde{k}\, dP_{0}.\label{vartheta_low}
\end{equation}
The ULAN property of $(P_{t_{n}})_{n\in\mathbb{N}}$ implies 
\begin{equation}
\left|P_{t_{n}}^{n}-P_{s_{n}}^{n}\right|\rightarrow0\,\mbox{ for }\, n\rightarrow\infty,\label{ulan}
\end{equation}
where $s_{n}:=n^{-\frac{1}{2}}t$, see Strasser(1985) and Hájek et.
al.(1999) for details. Let $\psi:(\Omega^{n},\mathcal{A}^{n})\rightarrow[0,1]$
be sequences of tests. For each subsequence $(n_{j})_{j\in\mathbb{N}}$
we obtain 
\begin{equation}
\lim_{j\rightarrow\infty}\int\psi\, dP_{t_{n_{j}}}^{n_{j}}=\lim_{j\rightarrow\infty}\int\psi\, dP_{s_{n_{j}}}^{n_{j}}=\lim_{j\rightarrow\infty}\int\psi\, dP_{n_{j},tg}\label{ulan_eq}
\end{equation}
if one of the limits exists. The last equality holds, since the asymptotical
distributions of $\frac{dP_{n,tg}}{dP_{n,0}}$ and $\frac{dP_{s_{n}}^{n}}{dP_{0}^{n}}$
under $P_{0}^{n}$ are equal, see Strasser (1985). Let $\psi_{1n}$
be a some sequence of the test for $\mathcal{H}_{1}$ against $\mathcal{K}_{1}$,
which have level $\alpha$ property asymptotically. By $(\ref{vartheta_low})$
we obtain, that $\psi_{1n}$ is a sequence of tests for the linear
test problem $\{h\in T(P_{0},\mathcal{P}):\int h\widetilde{k}\, dP_{0}\leq0\}$
against $\{h\in T(P_{0},\mathcal{P}):\int h\widetilde{k}\, dP_{0}>0\}.$
Lemma 82.6 from Strasser (1985) and (\ref{ulan_eq}) imply 
\[
\limsup_{n\rightarrow\infty}\int\psi_{1n}dP_{t_{n}}=\limsup_{n\rightarrow\infty}\int\psi_{1n}dP_{n_{j},tg}\leq\Phi\left(\frac{\vartheta}{\left\Vert \widetilde{k}\right\Vert }-u_{1-\alpha}\right)
\]
if $(P_{t_{n}})_{n\in\mathbb{N}}\in\mathcal{H}_{1}$ is an implicite
alternative and 
\[
\liminf_{n\rightarrow\infty}\int\psi_{1n}dP_{t_{n}}=\liminf_{n\rightarrow\infty}\int\psi_{1n}dP_{n_{j},tg}\geq\Phi\left(\frac{\vartheta}{\left\Vert \widetilde{k}\right\Vert }-u_{1-\alpha}\right)
\]
if $(P_{t_{n}})_{n\in\mathbb{N}}\in\mathcal{K}_{1}$ is an implicite
hypothese. Thus the asymptotic power functions of tests $\varphi_{1,n}$
is the upper bound for the asymptotic power functions of sequences
of tests for $\mathcal{H}_{1}$ against $\mathcal{K}_{1}$, which
have level $\alpha$ property asymptotically. This proves part $(a)$
of theorem \ref{one_two_sided_easy}. The proof of part $(b)$ of
theorem \ref{one_two_sided_easy} runs analogously under the application
of lemma 82.12 of Strasser (1985). \end{proof}


\begin{thebibliography}{10}
\bibitem{Bickel:1993} P. Bickel, C.A.J. Klaassen, Y. Ritov, J.A.
Wellner, \emph{Efficient and adaptive estimation for semiparametric
models.} Johns Hopkins University Press, Baltimore 1993.

\bibitem{Hajek:1999} J. Hájek, Z. Sidak, K.P. Sen, \emph{Theory of
rank tests}. Academic Press, 1999.

\bibitem{Ibragimov:81} I.A. Ibragimov, R.Z. Hasminskii, \emph{Statistical
estimation}. Springer, 1981.

\bibitem{LeCam:86} L. Le Cam, \emph{Asymptotic methods in statistical
decision theory}. Springer, 1986.

\bibitem{LeCam:00} L. Le Cam, G.L. Yang, \emph{Asymptotics in statistics:
some basic concepts}. Springer, 2000.

\bibitem{Janssen:99a} A. Janssen. \emph{Testing nonparametric statistical
functionals with applications to rank tests.} J. Statist. Plann. Infer.,
No. 81, 71-93, 1999.

\bibitem{Janssen:99b} A. Janssen. \emph{Nonparametric symmetry tests
for statistical functionals.} Math. Methods Stat., No.3, 320-343,
1999.

\bibitem{Janssen:00a} A. Janssen. \emph{Nonparametric bioequivalence
tests for statistical functionals and their efficient power functions.}
Statistics and Decisions, No. 18, 49-78, 2000.

\bibitem{Janssen:00b} A. Janssen. \emph{Global power functions of
goodness of fit tests.} Ann. Stat., No.1, 239-253, 2000.

\bibitem{Janssen:04} A. Janssen. \emph{Asymptotic relative efficiency
of tests at the boundary of regular statistical models.} J. Statist.
Plann. Infer., No. 126, 461-477, 2004.

\bibitem{Pfanzagl:82} J. Pfanzagl, W. Wefelmeyer. \emph{Contributions
to a general asymptotic statistical theory}, volume 13 of \emph{Lecture
Notes in Statistics}, Springer, 1982.

\bibitem{Pfanzagl:85} J. Pfanzagl, W. Wefelmeyer. \emph{Asymptotic
expansions for general statistical models}, volume 31 of \emph{Lecture
Notes in Statistics}. Springer, 1985.

\bibitem{Romano:05} P.J. Romano, \emph{Optimal testing of equivalence
hypotheses. }Ann. Stat. 33, No.3, 1036-1047, 2005.

\bibitem{Strasser:85} H. Strasser. Mathematical theory of statistics:
\emph{Statistical experiments and asymptotic decision theory}, volume
7 of\emph{ De Gruyter Studies in Mathematics}, de Gruyter, 1985.

\bibitem{Strasser:98} H. Strasser.\emph{ Differentiability of statistical
experiments.} Statistics and Decisions, No. 16, 113-130, 1998.

\bibitem{Vaart:85} A.W. van der Vaart. \emph{Statistical estimation
in large parameter spaces}, volume 44 of \emph{CWI Tract.} Center
for Mathematics and Computer Science, Amsterdam, 1985.

\bibitem{Vaart:91} A.W. van der Vaart. \emph{An asymptotic representation
theorem}, Int. Stat. Rev 59, No.1, 97-121, 1991.

\bibitem{Vaart:98} A.W. van der Vaart. \emph{Asymptotic statistics},
Cambridge Univ. Press., 1998.\end{thebibliography}
\end{document}